\documentclass[11pt,a4paper,twoside, abstracton , 
 headsepline]{scrartcl}
\usepackage{myscrartcl}
\usepackage{fouridx}
\usepackage{hyperref}
\usepackage{enumerate, enumitem}
\usepackage{subfigure}
\usepackage{amsfonts, amssymb, amsmath, amsbsy, esint} %
\usepackage{tikz}
\usepackage{array}

\usepackage{enumerate, enumitem}
\usepackage{epsfig} %

\usepackage{amsthm}

\theoremstyle{plain}
\newcounter{thm}[section]

\newtheorem{theorem}[thm]{Theorem}

\newtheorem{lemma}[thm]{Lemma}

\newtheorem{corollary}[thm]{Corollary}
\newtheorem{proposition}[thm]{Proposition}

\newcounter{ass}

\newtheorem{assumption}[ass]{Assumption}

\theoremstyle{definition}
\newcounter{exer}[section]

\newcounter{exam}[section]

\newcounter{defi}[section]

\newtheorem{definition}[defi]{Definition}

\newcounter{rem}[section]

\newtheorem{remark}[rem]{Remark}

\numberwithin{equation}{section} %
\newcommand{\wt}{\widetilde} %
\newcommand{\ov}{\overline} %
\newcommand{\wh}{\widehat} %
\newcommand{\la}{\lambda}
\newcommand{\wla}{\widehat\lambda}
\newcommand{\lar}{\lambda^r}
\newcommand{\wmu}{\widehat\mu}
\newcommand{\wxi}{\widehat\xi}
\newcommand{\wxo}{\widehat x_0}%
\newcommand{\wxd}{\widehat x_d}%
\newcommand{\wxf}{\widehat x_f}%
\newcommand{\wpo}{\widehat p_0}%
\newcommand{\wpf}{\widehat p_f}%

\newcommand{\wT}{\widehat T}
\newcommand{\wtau}{\widehat\tau}
\newcommand{\wbe}{\widehat\beta}
\renewcommand{\a}{\alpha} %
\renewcommand{\b}{\beta} %
\newcommand{\e}{\varepsilon} %
\newcommand{\N}{\mathbb N} %
\newcommand{\R}{\mathbb R} %
\newcommand\Rn{{\mathbb R}^n} %
\newcommand\Rns{\left({\mathbb R}^n\right)^*} %
\newcommand{\wS}[1]{\widehat S_{#1}} %
\newcommand{\wSinv}[1]{\widehat S^{-1}_{#1}} %
\newcommand{\Sr}[1]{S^r_{#1}} %
\newcommand{\Srinv}[1]{  (S^r_{#1})^{-1}  } %

\newcommand{\abs}[1]{\left\vert#1\right\vert} %
\newcommand{\norm}[1]{\left\Vert#1\right\Vert} %
\newcommand{\qo}{\text{a.e. }} %
\newcommand{\cO}{{\mathcal O}} %
\newcommand{\cU}{{\mathcal U}} %
\newcommand{\cV}{{\mathcal V}} %
\newcommand{\reftri}{\big( \wT, \wh\xi,\wh u \big)} 
\newcommand{\unoforma}{{\boldsymbol s}} %
\newcommand{\scal}[2]{\langle {#1} \, , \, {#2} \rangle} %
\newcommand{\dueforma}[2]{{\boldsymbol\sigma}\left( {#1}, {#2} \right) } %
\newcommand{\liebr}[2]{\left[  {#1}, {#2} \right] } %
\newcommand{\liede}[3]{  {#1} \cdot  {#2} \left( {#3}\right) } %
\newcommand{\liedede}[4]{  {#1} \cdot  {#2} \cdot{#3}\left( {#4}\right)  } %
\newcommand{\liededo}[3]{  {#1}^2 \cdot  {#2} \left( {#3}\right) } %

\newcommand{\bsi}{{\boldsymbol\sigma} } %
\newcommand{\ud}{\operatorname{d}\!} %
\newcommand{\da}{\operatorname{d}\!\alpha} %
\newcommand{\uD}{\operatorname{D}\!} %
\newcommand{\vFref}[1]{\overrightarrow{\wh F_{#1}}} %
\newcommand{\cFref}{\widehat{\mathcal F}} %
\newcommand{\vF}[1]{\overrightarrow{F_{#1}}} %
\newcommand{\vG}[1]{\overrightarrow G_{#1}} %
\newcommand{\vH}[1]{\overrightarrow{H_{#1}}} %
\newcommand{\vK}[1]{\overrightarrow{K_{#1}}} %
\newcommand{\vJ}[1]{\overrightarrow{J_{#1}}} %
\newcommand{\dds}{\displaystyle\frac{\ud}{\ud  s}} %
\newcommand{\ddt}{\displaystyle\frac{\ud}{\ud  t}} %
\newcommand{\dedt}[1]{\displaystyle\frac{\ud {#1}}{\ud  t}} %
\newcommand{\lo}{{\wh\ell_0}} %
\newcommand{\ld}{{\wh\ell_d}} %
\newcommand{\lf}{{\wh\ell_f}} %
\newcommand{\dl}{{\delta\ell}} %
\newcommand{\dt}{{\delta t}} %
\newcommand{\de}{{\delta e}} %
\newcommand{\df}{{\delta f}} %
\newcommand{\dx}{{\delta x}} %
\newcommand{\dep}{{\delta p}} %
\newcommand{\dxf}{{\delta x_f}} %
\newcommand{\dy}{\delta y} %
\newcommand{\se} {^{\prime \prime}} %
\newcommand{\Pbr}{\mathbf{P_r}}
\newcommand{\Pbz}{\mathbf{P_0}}

\newcommand{\ellor}{{\ell_0^r}} %
\newcommand{\mur}{\mu^r} 
\newcommand{\xir}{\xi^r} 
\newcommand{\Tr}{T^r}
\newcommand{\ur}{u^r}
\newcommand{\xik}{\xi^{r_k}}
\newcommand{\tauk}[1]{\tau^{r_k}_{#1}}
\newcommand{\Tk}{T^{r_k}}

\newcommand{\taur}[1]{\tau^r_{#1}}
\newcommand{\xor}{x_0^{r}} %
\newcommand{\Nor}{N_0^{r}} %
\newcommand{\Nfr}{N_f^{r}} %

\newcommand{\Phior}[1]{
\fourIdx{}{}{0,r}{{#1}}{\Phi} } %

\newcommand{\Phio}[1]{
\fourIdx{}{}{0,0}{{#1}}{\Phi} } %

\newcommand{\Psior}[1]{
\fourIdx{}{}{}{{#1}}{\Psi} } %

\newcommand{\ar}{\alpha^{r}}

\newcommand{\Phifr}[1]{ \fourIdx{}{}{f,r}{{#1}}{\Phi} } %
\newcommand{\Phif}[1]{ \fourIdx{}{}{f,0}{{#1}}{\Phi} } %
\newcommand{\br}{\beta^{r}}
\newcommand{\fr}[1]{f_{#1}^r}
\newcommand{\gr}[1]{g_{#1}^r}
\newcommand{\jr}[1]{j_{#1}^r}
\newcommand{\hr}[1]{h_{#1}^r}
\newcommand{\kr}[1]{k_{#1}^r}
\newcommand{\Fr}[1]{F_{#1}^r}
\newcommand{\Hr}[1]{H_{#1}^r}
\newcommand{\Kr}[1]{K_{#1}^r}
\newcommand{\vFr}[1]{\overrightarrow{F_{#1}^r}} %
\newcommand{\vHr}[1]{\overrightarrow{H_{#1}^r}} %
\newcommand{\vKr}[1]{\overrightarrow{K_{#1}^r}} %

\newcommand{\Span}{\operatorname{span}} 
\newcommand{\ball}{B_R}
\newcommand{\ballt}{B_{\widetilde R}}

\newcommand\pd[2]{\displaystyle\frac{\partial#1}{\partial#2}}

\title{Structural stability of bang--bang trajectories with a double switching time in the
  minimum time problem}

\author{Laura Poggiolini} %
 \date{\today
}

\allowdisplaybreaks

\begin{document}
\maketitle

\begin{abstract}
 In this paper we consider the problem of structural stability of strong local optimisers for the minimum time problem in the case when the nominal problem has a bang-bang strongly local optimal control which exhibits a double switch.
\end{abstract}

%
%

\section{Introduction}
In this paper we consider the minimum time problem between two submanifolds
of a finite dimensional manifold $M$ in the case when the dynamics is affine with respect to the control and the control takes values in a
box of $\R^m$.  Namely, the following optimal control is studied:
\begin{subequations}
  \label{eq:problema}
  \begin{alignat}{2}
    & T \to \min, 
    \\
    & \dot\xi(t) = f_{0}(\xi(t)) 
    + \displaystyle\sum_{s=1}^m u_{s}(t) f_{s}(\xi(t))\quad && \qo t
    \in [0, T], \label{eq:dinamica} \\
    & \xi(0) \in N_0, \qquad \xi(T) \in N_f, \label{eq:estremi} \\
    & \abs{u_{s}(t)} \leq 1 \quad s = 1,2,\ldots, m && \qo t \in [0, T]. \label{eq:box}
\end{alignat}
\end{subequations}
For such a problem, the triple $(T, \xi, u)$ is said to be an {\em admissible
  triple} for problem \eqref{eq:problema} if $T > 0$ and the couple $(\xi, u) \in W^{1, \infty}([0, T], M) \times L^\infty([0, T], \R^m)$ satisfies \eqref{eq:dinamica},
\eqref{eq:estremi} and \eqref{eq:box}. 
We assume that a triple $(\wh T, \wh\xi, \wh u)$ satisfying the necessary conditions for optimality (i.e.~Pontryagin Maximum Principle) is given where the control is bang-bang but multiple switches occur. 
To the author's knowledge the literature on bang-bang controls with multiple switches is much more scarce than the one with simple switches only. $L^1$-local optimality results for bang-bang controls with multiple switches in the minimum time problem between two fixed end points were given in \cite{Sar97}. In \cite{PSp15} the authors consider the case when a double switch occurs and all the other switches are simple. They prove that under suitable regularity conditions, and assuming the coercivity of the second order approximation of a certain finite-dimensional subproblem of the given one, the triple   $(\wh T, \wh\xi, \wh u)$ is in fact a state-local minimiser of the problem (see Definition \ref{def:stateloc} for a precise definition of this kind of strong local optimality).

Here we consider the same case as in \cite{PSp15} and we study the structural stability of the locally optimal control $\wh u$ under smooth perturbations of the data of the problem, namely the drift $f_0$, the controlled vector fields $f_1$, $f_2$, \ldots, $f_m$ and the submanifolds of the initial and final constraints.

In particular we are interested in understanding how the existence of the double switch and the bang-bang structure of the locally optimal control are affected by small perturbations of the data. Such a situation is in fact not generic and we show here that under the same assumptions that ensure state-local optimality of the reference triple plus a controllability assumption, the bang-bang structure of the locally optimal control is stable under small perturbations even though the double switching time may decouple into two simple switching times.

The proof is carried out by Hamiltonian methods, which were also used in \cite{PSp15} to prove the state local optimality result for the nominal problem. The same methods were also used in \cite{PS11} and \cite{PS13b} to prove state local optimality and structural stability of a bang-singular-bang extremal in the minimum time problem between two fixed end points.
The same methods were used in \cite{PSp11} and \cite{FPS09} for the problem of strong local optimality and structural stability of bang-bang extremals with a double switch in Mayer problem.

As in \cite{PSp15}, for the sake of notational simplicity we shall confine ourselves to the case when $M=\Rn$, $m=2$ and only the double switch occurs. However, as all the results are invariant under a change of coordinates, they can be easily generalised to the case when the state
space is a smooth finite dimensional manifold. Moreover, the presence of a finite number of simple switches occuring either before and/or after the double one can be treated at the expenses of a much heavier notation, see for example \cite{PSp11}.
Thus the nominal problem \eqref{eq:problema} simplifies to
\begin{alignat}{2}
    & T \to \min,   \label{eq:problema0}\tag{$\Pbz$} \\ 
    & \dot\xi(t) = f_{0}(\xi(t)) 
    + u_{1}(t) f_{1}(\xi(t)) + u_{2}(t) f_{2}(\xi(t)) \quad && \qo t
    \in [0, T], \notag \\ 
    & \xi(0) \in N_0, \qquad \xi(T) \in N_f, \notag \\ 
    & \abs{u_{s}(t)} \leq 1 \quad s = 1,2 && \qo t \in [0, T].  \notag
\end{alignat}
Without loss of generality  we can assume that $\wh u$
is given by
\begin{equation*}
  \wh u(t) =
  \left(     \wh u_1(t) , \  \wh u_2(t) \right) =
  \begin{cases}
    (-1, -1) \quad & t \in [0, \wtau) , \\
    (1, 1) \quad & t \in (\wtau, \wT].
  \end{cases}
\end{equation*}
We assume that \eqref{eq:problema0} is the problem we obtain when $r=0$ in the following parameter dependent problem \eqref{eq:problemar}:
  \begin{alignat}{2}
    & T \to \min,  \label{eq:problemar}\tag{$\Pbr$} \\ 
    & \dot\xi(t) = \fr{0}(\xi(t)) 
    + u_{1}(t) \fr{1}(\xi(t)) + u_{2}(t) \fr{2}(\xi(t)) \quad && \qo t
    \in [0, T], \notag \\ 
    & \xi(0) \in \Nor, \qquad \xi(T) \in \Nfr,  \notag\\ 
    & \abs{u_{s}(t)} \leq 1 \quad s = 1,2 && \qo t \in [0, T]. \notag 
  \end{alignat}
The parameter $r$ belongs to some ball $\ball$ centered at the origin of $\R^k$ and radius $R >0$.
For notational simplicity we choose $\R^n$ as state--space;
 All the data are assumed to be smooth, more precisely the maps
\begin{equation*}
(r,x) \in \ball \times \R^n \mapsto \fr{i}(x) \in \R^n, \qquad i =0,1, 2
\end{equation*}
are assumed to be $C^2$ and the submanifolds of the initial and final constraints are given as regular intersections of zero-level sets of $C^2$ maps from $\ball \times \Rn$ to $\R$, i.e.
\begin{equation*}
\begin{split}
	\Nor  \colon & \Phior{i}(x) = 0 \qquad \forall i= 1, \ldots,
      n-n_0 ,  \\
& \uD\Phior{i}(x) \text{ are linearly independent at } x \quad \forall (r,x) \in \ball\times \Rn.
\end{split}
\end{equation*} %
and 
\begin{equation*}
\begin{split}
	\Nfr \colon & \Phifr{j}( x) = 0 \qquad \forall j= 1, \ldots,
      n-n_f, \\
      & \uD\Phifr{j}(x) \text{ are linearly independent at } x \quad \forall (r,x) \in \ball\times \Rn.
\end{split}
\end{equation*}%
We are interested in state-local optimisers according to the following definition:
\begin{definition}[state-local optimality] \label{def:stateloc} 
  The trajectory $\xi$ of an admissible triple $(T, \xi, u)$ for problem \eqref{eq:problemar}  is a state-local minimiser of such problem if there are
  neighbourhoods $\cU$ of its range $\xi([0, T])$, $\cU_0$ of
  $\xi(0)$ and $\cU_f$ of $\xi(T)$ such that $\xi$ is a minimum time trajectory   among the admissible trajectories of \eqref{eq:problemar} whose range is in $\cU$, whose
  initial point is in $\Nor \cap \cU_0$ and whose final point is in
  $\Nfr \cap \cU_f$.
\end{definition}
\begin{remark}
Notice that state-local optimality is a kind of strong local optimality, in the sense that there is no localisation with respect to the control, but only with respect to the trajectories. Moreover state-local optimality is stronger than the classical notion of strong-local optimality where one considers the $C^ 0$ distance between trajectories, i.e.~one considers only triples $\left( T, \xi, u \right)$ where the graph of the trajectory $\xi$ is close to the graph of the reference trajectory $\wxi$.
\end{remark}
Assuming that $\reftri$ satisfies normal PMP, the sufficient conditions for state-local optimality as stated in \cite{PSp15} and a controllability assumption which ensures the uniqueness of the adjoint covector, we prove that for small $R$ each problem \eqref{eq:problemar}, $r \in \ball$ has a state-local optimal trajectory $(\Tr, \xir, \ur)$ (with adjoint covector $\lar$) where $\ur$ preserves the bang-bang structure of $\wh u$ and $\Tr$ is close to $\wh T$. Moreover $\lar$ is the only Pontryagin extremal of \eqref{eq:problemar} whose graph is close to the graph of $\wla$.

\section{Notation}
\label{sec:notation}
We are going to use some basic notions from symplectic geometry.  For
any manifold $N \subset \Rn$ and any $x \in N$, the tangent space and
the cotangent space to $N$ in $x$ are denoted as $T_xN$ and $T^*_xN$,
respectively.  We recall that the cotangent bundle $T^*\Rn$ to $\Rn$
can be identified with the Cartesian product $\Rns \times \Rn = T^*_x\Rn \times T_x\Rn$ for any $x \in \Rn$.  The
projection from $T^*\Rn$ onto $\Rn$ is denoted as $\pi \colon \ell \in
T^*\Rn \mapsto \pi\ell \in \Rn$.  We shall write $T_x\R^n$ instead of $\R^n$, to emphasize the
fact that we are dealing with tangent vectors.%


The canonical Liouville one--form $\unoforma 
$ on $T^*\Rn$ and the associated canonical symplectic two-form $\bsi =
\ud\unoforma$ allow to associate to any, possibly time-dependent,
smooth Hamiltonian $F_t \colon T^*\Rn \rightarrow \R$, the unique
Hamiltonian vector field $\vF{t}$ such that
\begin{equation*}
\bsi(v,\vF{t}(\ell))=\scal{ \ud F_t(\ell)}{v} ,\quad \forall v\in
T_{\ell}T^*\Rn .
\end{equation*}
Choosing coordinates $\ell = (p, x) \in \Rns \times \Rn$, we have
\begin{equation*}
\vF{t}(p,x) = \left( \dfrac{- \, \partial F_t}{\partial x},
  \dfrac{\partial F_t}{\partial p} \right)(p,x).
\end{equation*}
To any vector field $f \colon \Rn \to T\Rn$ we associate the Hamiltonian
function $F$ 
\begin{equation*}
F \colon \ell \in T^*\Rn \mapsto \scal{\ell}{f(\pi\ell)} \in \R,
\end{equation*}
so that
  $\vF{}(p,x) = \big( -p \ud f(x), f(x) \big)$.
  
We denote by $ \wh{f_{t}} $ the piecewisely time-dependent vector field
associated to the reference control:
\begin{equation*}
\wh {f_{t}} := f_{0} + \wh u_{1}(t) f_{1} + \wh u_{2}(t) f_{2}
\end{equation*}
and by $h_1$, $h_2$ its restrictions to the time intervals $[0,
\wtau)$ and $(\wtau, \wT]$, respectively:
\begin{equation*}
h_{1} := \left. \wh {f_{t}} \right\vert_{[0, \wtau)} = f_{0} - f_{1} -
f_{2} , \qquad h_{2} := \left. \wh {f_{t}} \right\vert_{(\wtau, \wT]}
= f_{0} + f_{1} + f_{2} .
\end{equation*}
In what follows we shall also neeed the vector fields
\begin{equation*}
\begin{split}
  & k_{1} := f_{0} + f_{1} - f_{2} = h_{1} + 2 f_{1} = h_{2} - 2 f_{2} , \\
  & k_{2} := f_{0} - f_{1} + f_{2} = h_{1} + 2 f_{2} = h_{2} - 2 f_{1}
  .
\end{split}
\end{equation*}
The associated Hamiltonian functions are denoted by the same letter,
but capitalized. Namely
 \begin{alignat*}{2}
   & H_1(\ell) := \scal{\ell}{h_1(\pi\ell)}, \qquad && H_2(\ell) := \scal{\ell}{h_2(\pi\ell)}, \\
   & K_1(\ell) := \scal{\ell}{k_1(\pi\ell)}, \qquad && K_2(\ell) :=
   \scal{\ell}{k_2(\pi\ell)}.
 \end{alignat*}
Analougously we define the parameter dependent vector fields
\begin{alignat*}{2}
&\hr{1} := \fr{0} - \fr{1} - \fr{2} , \qquad && \hr{2} := \fr{0} + \fr{1} + \fr{2} , \\
&\kr{1} := \fr{0} + \fr{1} - \fr{2} , \qquad && \kr{2} := \fr{0} - \fr{1} + \fr{2} ,
\end{alignat*}
and the associated parameter dependent Hamiltonians
\begin{alignat*}{2}
&\Hr{1} := \Fr{0} - \Fr{1} - \Fr{2} , \qquad && \Hr{2} := \Fr{0} + \Fr{1} + \Fr{2} , \\
&\Kr{1} := \Fr{0} + \Fr{1} - \Fr{2} , \qquad && \Kr{2} := \Fr{0} - \Fr{1} + \Fr{2} .
\end{alignat*}

The maximised Hamiltonian of the nominal control system \eqref{eq:problema0} is
well defined in the whole cotangent bundle $T^*\R^n$ and is denoted by
$H^{\max}$:
\begin{equation*}
\begin{split}
H^{\max}(\ell) := & \max \left\{ F_0(\ell) + u_1 F_1(\ell) + u_2
  F_2(\ell) \colon (u_1, u_2) \in [-1, 1]^2 \right\} \\ 
   = & F_0(\ell) + \abs{F_1(\ell)} + \abs{F_2(\ell)}.
\end{split}
\end{equation*}
Throughout the paper, the symbol $\cO(x)$ denotes a neighborhood of
$x$ in its ambient space.  The flow starting at time $t = 0$ of the
time-dependent vector field $ \wh {f_{t}}$ is defined in a
neighborhood $\cO(\wxo)$ for any $t\in[0, \wT]$ and is denoted by $\wS{t} \colon \cO(\wxo) \to \Rn$, i.e.~
\begin{equation*}
\ddt \wS{t}(x) = \wh {f_{t}}\circ\wS{t}(x) \quad \qo t \in [0,
\wT], \qquad \wS{0}(x) = x.
\end{equation*}
We denote by $\wxo := \wxi(0)$ and by $\wxf := \wxi(\wT) = \wS{T}(\wxo)$ the end points of the reference trajectory and by $\wxd := \wxi(\wtau) = \wS{\wtau}(\wxo)$ the point corresponding to the switching time. 

Given a smooth function $\gamma \colon \cO(x) \subset \Rn \to \R$ and a vector $\dx \in T_x\Rn$, the Lie derivative of $\gamma$ with respect to the vector $\dx$ at the point $x$ is denoted by 
$\liede{\dx}{\gamma}{x}$, i.e.~$\liede{\dx}{\gamma}{x} = \scal{\uD\gamma(x)}{\dx}$. If $f \colon \cO(x) \to T\Rn$ is a smooth vector field, then $\liede{f}{\gamma}{x}$ is the Lie derivative of $\gamma$ at $x$ with respect to the vector $f(x)$, i.e.~$\liede{f}{\gamma}{x}  := \scal{\uD\gamma(x)}{f(x)}$.

Finally, given two smooth vector fields $f, \ g \colon \Rn \to T\Rn$, then the Lie bracket $\liebr{f}{g}$ is given by the vector field $(\uD g) f - (\uD f) g$.
\section{Assumptions}
\label{sec:assumptions}
We now state the assumptions on the nominal extremal triple $\reftri$ of \eqref{eq:problema0}. Besides the necessary conditions for optimality, namely Pontryagin Maximum
Principle (PMP) --which we assume to hold in its normal form-- we require that the triple $\reftri$ satisfies the conditions that ensure state-local optimality, as stated in \cite{PSp15}: regularity along the bang arcs, regularity at the switching time and the coercivity of the second order variation associated to some finite-dimensional subproblem of the given one. 
Moreover we assume that the nominal problem \eqref{eq:problema0} is controllable along $\wxi$.
\begin{assumption}[Normal PMP]
  \label{ass:PMP}
There exists an absolutely continuous curve $\wla\colon [0, \wT] \to T^*\Rn$ satisfying the following properties
\begin{subequations}
  \begin{alignat}{2}
    & \pi\wla(t) = \wxi(t), \qquad &&\forall t \in [0, \wT], 
    \displaybreak[0] \\
    \qquad
    &\dot{\wla}(t) = \overrightarrow{\wh F}_t(\wla(t)), \qquad &&\qo t \in [0, \wT], 
      \displaybreak[0] \\
    & \wh F_t (\wla(t)) = H^{\max}(\wla(t)) = 1, \quad &&\qo t
    \in [0, \wT],     \label{eq:maxi} \displaybreak[0] \\
    & \left. \wla(0)\right\vert_{T_{\wxo}N_0} = 0 , \quad 
\left. \wla(\wT)\right\vert_{ T_{\wxf}N_f}  = 0  . \qquad &&  \label{eq:trans} 
  \end{alignat}
  \end{subequations}
\end{assumption}
In coordinates we put  $\wla(t) := \left( \wmu(t), \wxi(t) \right)$ where $\wmu(t) \in T^*_{\wxi(t)}\Rn \quad \forall t \in [0, \wh T]$.

Here and in what follows we shall use the following notation:
\begin{alignat*}{3}
& \lo := \wla(0), \quad 
&& \ld := \wla(\wtau), \quad 
&& \lf := \wla(\wT), \\
& \wpo := \wmu(0),  \quad
&& \wh p_d := \wmu(\wtau), \quad 
&& \wpf := \wmu(\wT).
\end{alignat*} %

\begin{remark}
The flow starting at time $t = 0$ of the time-dependent Hamiltonian vector field associated to $\wh {F_{t}}(\ell) := \scal{\ell}{\wh {f_t}(\pi\ell)}$ is defined in a neighborhood $\cO(\lo)$ of $\lo$ for any $t \in [0, \wh T]$ and is denoted by $\cFref_t \colon \cO(\lo) \to T^*\Rn$:
\begin{equation*}
\ddt \cFref_t(\ell) = \vFref{t} \circ\cFref_t(\ell) \quad \qo t \in [0, \wh T], \qquad \cFref_0(\ell) = \ell.
\end{equation*}
\end{remark}
\begin{remark}\label{rem:mu}
The adjoint covector $\wmu$ is a solution to the ODE
\begin{equation*}
\dot\mu(t) =   \dfrac{- \, \partial F_t}{\partial x}\left( \mu(t), \wxi(t) \right)
= - \, \scal{\mu(t)}{\ud \wh f_t(\wxi(t))}
\end{equation*}
so that $\wmu(t) = \wpo \wSinv{t\, *} \quad \forall t \in [0, \wT]$ and equations \eqref{eq:trans} read
\begin{equation*}
\wpo = \sum_{i=1}^{n-n_0} \wh a_i \uD \Phio{i}(\wxo), \qquad 
\wpf = \sum_{j=1}^{n-n_f} \wh b_j \uD \Phif{i}(\wxf), \qquad 
\end{equation*}
for some $\wh a = \begin{pmatrix}
\wh a_1, \ldots, \wh a_{n-n_0}
\end{pmatrix} \in \R^{n-n_0}$, $\wh b = \begin{pmatrix}
\wh b_1, \ldots, \wh b_{n-n_f}
\end{pmatrix} \in \R^{n-n_f}$.
\end{remark}
\begin{remark}\label{rem:PMP}
  As $\wla$ is a normal extremal 
  then the
  transversality conditions \eqref{eq:trans} together with the
  maximality condition \eqref{eq:maxi} yield $ h_1(\wxo) \notin
  T_{\wxo} N_0$ and $h_2(\wxf) \notin T_{\wxf} N_f$.
\end{remark}
Maximality condition \eqref{eq:maxi} implies, for any $i=1,2$ and for almost every $t\in [0, \wT]$, 
\begin{equation*}\wh u_i(t) F_i(\wla(t)) = \wh u_i(t) \scal{\wla(t)}{f_i(\wxi(t))} \geq 0.
\end{equation*}%
We assume that the bang arcs of $\wla$ are regular, i.e., we assume
that at each point $\wla(t)$, $t \neq \wtau$, the maximum of the
Hamiltonian is achieved only by $u= \wh u(t) = \left(
\wh u_1(t), \wh u_2(t)\right)$, i.e.,
\begin{multline*}
 F_0(\wla(t)) + u_1 F_1(\wla(t)) + u_2 F_2(\wla(t)) < H^{\max}(\wla(t))  = 1
\\ %
\forall (u_1, u_2) \in [-1,1]^2 \setminus\{ (\wh u_1(t), \wh
u_2(t)) \}.
\end{multline*}
In terms of the controlled Hamiltonians $F_1$ and $F_2$ this can be
stated as follows:
\begin{assumption}[Regularity along the bang arcs] \label{ass:1} Let
  $i = 1, 2$. If $t \neq \wh \tau$, then
\begin{equation}  \label{eq:reg}
\wh{u}_i(t) F_i(\wla(t)) = \wh{u}_i(t) \scal{\wla(t)}{f_i(\wxi(t))} > 0.
  \end{equation}
\end{assumption}
\begin{remark}
\label{re:normale}
Passing to the limit for $t \to \wtau$ in \eqref{eq:reg} we get $F_1(\ld) = F_2(\ld) = 0$ and, because of the normality condition in PMP, $F_0(\ld) = 1$. As a consequence $f_0(\wxd) \notin \Span\{f_1(\wxd), \ f_2(\wxd) \}$.
\end{remark}
\noindent From the necessary maximality condition \eqref{eq:maxi} we
get
\begin{equation*}
  \begin{split}
    & 
    \left.\dfrac{\ud}{\ud t} 2 \, F_i \circ\wla(t)\right\vert_{t =
      \wtau - } \!\!\!\! = \left.\dfrac{\ud}{\ud t}\left( K_i -
        H_{1}\right) \circ\wla(t)\right\vert_{t = \wtau - }
    \!\!\!\! \geq 0, \\
    & 
    \left.\dfrac{\ud}{\ud t} 2 \, F_i \circ\wla(t)\right\vert_{t =
      \wtau + } \!\!\!\! = \left.  \dfrac{\ud}{\ud t}\left( H_{2} -
        K_i \right)\circ\wla(t) \right\vert_{t = \wtau + } \!\!\!\!
    \geq 0,
  \end{split} \qquad i =1, 2 .
\end{equation*}
We assume that the above inequalities are strict:
\begin{assumption}[Regularity at the double switching time]
  \label{ass:3}
  \begin{equation*} 
    \begin{split}
      & \left.\dfrac{\ud}{\ud t}\left( K_\nu - H_{1}
        \right)\circ\wla(t)\right\vert_{t = \wtau - } \!\!\!\!  > 0,
      \qquad \left.  \dfrac{\ud}{\ud t}\left( H_{2} - K_\nu
        \right)\circ\wla(t) \right\vert_{t = \wtau + } \!\!\!\!  > 0,
    \end{split} \quad \nu =1, 2.
  \end{equation*}
\end{assumption}
Assumption \ref{ass:3} is called the {\sc Strong bang-bang Legendre condition for the double switching time}.
Equivalently, this assumption can be expressed in terms
of the Lie brackets of vector fields or in terms of the canonical
symplectic structure $\dueforma{\cdot}{\cdot}$ on $T^*\Rn$.

\begin{proposition}
  Assumption \ref{ass:3} is equivalent to
  \begin{equation*} 
    \begin{split}
      & \scal{\ld}{\liebr{h_{1}}{k_\nu} (\wxd)} =
      \dueforma{\vH{1}}{\vK{\nu}}(\ld)  > 0, \\
      & \scal{\ld}{\liebr{k_\nu}{ h_{2}} (\wxd)} =
      \dueforma{\vK{\nu}}{\vH{2}}(\ld) > 0,
    \end{split}
    \qquad  \nu = 1,2 . 
  \end{equation*}
\end{proposition}
An easy computation  proves the following equivalent condition
\begin{proposition}\label{prop:ass3eq}
Assumption \ref{ass:3}  is equivalent to   
\begin{equation} 
	\scal{\ld}{\liebr{f_0}{f_i} (\wxd)} > \abs{\scal{\ld}{\liebr{f_1}{ f_2} (\wxd)}} ,  \qquad  i = 1,2, \label{eq:absleg2}	
\end{equation}
i.e.
\begin{equation*} 
    \dueforma{\vF{0}}{\vF{i}}(\ld) > \abs{\dueforma{\vF{1}}{\vF{2}}(\ld)} , \qquad  i = 1,2 . 
\end{equation*}
\end{proposition}
In what follows we shall also need to reformulate Assumption
\ref{ass:3} in terms of the pull-backs of the vector fields $h_\nu$ and $k_\nu$ along the reference flow $\wS{t}$ . Define
\begin{equation}
  g_\nu(x) 
  := \wSinv{\wtau \, *} h_\nu  \circ \wS{\wtau }(x),
  \quad
  j_\nu (x) := \wSinv{\wtau\, *} k_\nu \circ \wS{\wtau}(x),
  \qquad \nu =1,2   \label{eq:pullback}
\end{equation}
and let $G_\nu$, $J_\nu$ be the associated Hamiltonians. Then a
straightforward computation yields
\begin{proposition}
  Assumption \ref{ass:3} is equivalent to
  \begin{equation}\label{eq:leg2c}
    \begin{split}
      & \scal{\lo}{ \liebr{g_{1}}{j_\nu} (\wxo)} =
      \dueforma{\vG{1}}{\vJ{\nu}}(\lo)  > 0, \\
      & \scal{\lo}{ \liebr{j_\nu}{ g_{2}} (\wxo)} =
      \dueforma{\vJ{\nu}}{\vG{2}}(\lo) > 0,
    \end{split}
    \qquad  \nu = 1,2 . 
  \end{equation}
\end{proposition}
Also, we assume that $\wxi$ has no self-intersection:
\begin{assumption}\label{ass:inje}
The reference trajectory $\wxi \colon [0, \wT ] \to \Rn$ is injective.
\end{assumption}
\section{The second order variation}
\label{sec:secvar}
The second order variation is the second order approximation of a  finite-di\-men\-sio\-nal sub-problem of
\eqref{eq:problema0} obtained by keeping the same end-point constraints and
restricting the set of admissible controls. Namely, we allow for
independent variations of the switching times of each of the two
reference control components $\wh u_1$ and $\wh u_2$.
This sub-problem is then extended by allowing for variations of the
initial points of trajectories on a neighborhood of $\wxo$ in
$\Rn$. We penalise the latter variations with a smooth cost $\a$ that
vanishes on $N_0$.

We allow for perturbations of the final time, of the initial point of
trajectories on $N_0$, of the final point on $N_f$ and of the
switching time of either component of the reference control: let
$\tau_{1} := \wtau + \e_{1}$ and $\tau_{2} := \wtau + \e_{2}$ be the
perturbed switching times of the first and of the second component of
$\wh u$, respectively, and let $\tau_{3} := \wT + \e_{3}$ be the
perturbation of the final time $\wT$. 

Let $\a \colon \Rn \to \R$ be a smooth nonnegative function vanishing
on $N_0$. We remove the constraint on the initial point $\xi(0)$
introducing the penalty cost $\a$ on such point. We thus obtain the
following problem in the unknowns $x$, $\e_1$, $\e_2$, $\e_3$:
\begin{subequations}\label{eq:prob4}
  \begin{align}
    & \a(x) + \wT + \delta_{3} \to
    \min, 
    \\
    & \dot\xi =
    \begin{cases}
      h_{1}(\xi(t)) \quad & t \in (0, \wtau + \delta_{1}), \\
      k_{\nu}(\xi(t)) \quad & t \in ( \wtau + \delta_{1}, \wtau + \delta_{2}), \\
      h_{2}(\xi(t)) \quad & t \in ( \wtau + \delta_{2},
      \wT + \delta_{3}),
    \end{cases} 
    \\
    & \xi(0) = x \in \Rn, \qquad \xi(\wT + \delta_{3}) \in N_{f}, 
    \\
    & \delta_{1} := \min \{\e_{1}, \e_{2}\}, \quad \delta_{2} := \max
    \{\e_{1}, \e_{2}\}, \quad \delta_{3} := \e_{3}, \\
    & \nu =
    \begin{cases}
      1 \quad &\text{if } \e_{1} \leq \e_{2}, \\
      2 \quad &\text{if } \e_{1} > \e_{2}.
    \end{cases}
  \end{align}
\end{subequations}
%
%
Let $g_\nu$, $j_\nu$, $\nu =1,2$ be the pullbacks along the reference
flow of the vector fields $h_\nu$ and $k_\nu$, as defined in equation
\eqref{eq:pullback}. %
Let $\wh N_f$ be the pullback of $N_f$ to time $t= 0$ along the reference flow:
\begin{equation*}
  \wh N_f := \wSinv{\wT}(N_f)
\end{equation*}
and let $T_{\wxo} \wh N_f = \wSinv{\wT\, *}(T_{\wxf}N_f)$
be its tangent space at $\wxo$. \\
By the transversality condition \eqref{eq:trans} at the reference
final time $\wT$, there exists a smooth function $\b \colon \Rn \to
\R$ that vanishes on $N_f$ and such that $\ud\b(\wxf) = -\lf$. Also
let $\wbe$ be the pull-back of $\b$ along the reference flow, $\wbe :=
\b \circ \wh{S}_{\wT}$ so that, by Remark \ref{rem:mu},
\begin{equation*}
\wbe \colon \cO(\wxo) \to \R, \quad \left. \wbe \right\vert_{\cO(\wxo)o  \cap \wh N_f} \equiv 0, \qquad \ud\wbe(\wxo) = -\wpo.
\end{equation*}
Let us set
\begin{equation*}
a_{1} := \delta_{1}, \qquad b := \delta_{2} - \delta_{1} = \abs{\e_{2}
  - \e_{1}} , \qquad a_{2} := \delta_{3} - \delta_{2};
\end{equation*}
then the second order approximations of problem \eqref{eq:prob4}, for
$\nu=1, 2$, are defined on the closed half-spaces
\begin{multline*}
\qquad V^{+}_{\nu} := \left\{ (\dx, a_{1}, b, a_{2}) \in \Rn \times \R \times
  \R^{+} \times \R \colon \right. \\ 
  \left. \dx + a_{1}g_{1}(\wxo) + b \, j_{\nu}(\wxo)
  + a_{2}g_{2}(\wxo) \in T_{\wxo} \wh N_f \right\}
  \qquad 
\end{multline*}

and are given by
\begin{equation}
  \label{eq:secvar0}
  \begin{split}
    J\se_{\nu} & [\dx, a_{1}, b, a_{2}] = \uD^{2} ( \a + \wh
    {\b})(\wxo)[\dx]^{2} + 2 \, \dx \cdot (a_{1}g_{1} + b \, j_{\nu} +
    a_{2}g_{2})\cdot \wh {\b}(\wxo)  \\ 
    & + (a_{1}g_{1} + b \, j_{\nu} +
    a_{2}g_{2})^{2}\cdot\wh {\b}(\wxo)
    \\
    & + a_{1}b \, \liebr{g_{1}}{j_{\nu}} \cdot \wh {\b}(\wxo) +
    a_{1}a_{2} \liebr{g_{1}}{ g_{2}} \cdot \wh {\b}(\wxo) + b \, a_{2}
    \liebr{j_{\nu}}{ g_{2}} \cdot \wh {\b}(\wxo),
  \end{split}
\end{equation}
see \cite{PSp11} for the construction.  The restrictions of
$J\se_{\nu}$ to the sets
\begin{multline*}
V_{0, \, \nu}^{+} := \left\{ (\dx, a_{1}, b, a_{2}) \in T_{\wxo} N_0
  \times \R \times \R^{+} \times \R \colon \right. \\ 
  \left.  \dx + a_{1}g_{1}(\wxo) + b
  \, j_{\nu}(\wxo) + a_{2}g_{2}(\wxo) \in T_{\wxo} \wh N_f \right\},
\qquad \nu=1,2,
\end{multline*}
are indeed the second order approximation of \eqref{eq:problema0}.

We are now in a position to state our  assumption on the second order approximation of sub-problem \eqref{eq:prob4}.
\begin{assumption}\label{ass:coerc}
  For each $\nu = 1, 2$, $J\se_\nu$ is coercive on $V^+_{0, \nu}$.
\end{assumption}
\noindent Since both $J\se_{1}$ and $J\se_{2}$ are quadratic forms, we
may as well remove the constraint $b \geq 0$ and let them be defined and coercive
on the linear spaces
\begin{multline}
  \label{eq:Vnu}
  V_{0,\nu} := \left\{ (\dx, a_{1}, b, a_{2}) \in T_{\wxo}N_0 \times \R^{3} \colon \right. \\ 
\left.     \dx + a_{1}g_{1}(\wxo) + b \, j_{\nu}(\wxo) + a_{2}g_{2}(\wxo) \in
    T_{\wxo}\wh N_f \right\}, \qquad \nu =1, 2.  
\end{multline}
Also let
\begin{multline}
V_{\nu} := \left\{ (\dx, a_{1}, b, a_{2}) \in \Rn \times
  \R^{3}  \colon \right. \\ 
  \left. \dx + a_{1}g_{1}(\wxo) + b \, j_{\nu}(\wxo) +
  a_{2}g_{2}(\wxo) \in T_{\wxo}\wh N_f \right\}, \qquad \nu =1, 2.
\end{multline}
By \cite{Hes51} we obtain the following:
\begin{theorem}\label{thm:coercalpha}
  If the second order approximations $J\se_{1}$ and $J\se_{2}$
  are coercive on $V_{0, \, 1}$ and $V_{0, \, 2}$ respectively, then
  there exists a smooth function $\a \colon \Rn \to \R$ such that
  $\left. \a \right\vert_{N_0} \equiv 0$, $\da(\wxo) =\lo$ and both
  $J\se_{1}$ and $J\se_{2}$ are coercive quadratic forms on $V_{1}$
  and $V_{2}$, respectively.
\end{theorem}
\noindent The main result of \cite{PSp15} is the following:
\begin{theorem}
  \label{thm:main}
  Assume $\reftri$ is an admissible triple for the minimum time
  problem \eqref{eq:problema}. Assume the triple is bang-bang with
  only one switching time which is a double switching time. Assume the
  triple satisfies PMP, the regularity
  assumption along the bang arcs (Assumption \ref{ass:1}), the
  regularity assumption at the double switching time (Assumption
  \ref{ass:3}) and the coercivity assumption (Assumption
  \ref{ass:coerc}). Moreover assume the trajectory $\wxi$ is
  injective.  Then, $\wxi$ is a strict state-locally optimal
  trajectory.
\end{theorem}
%

\section{The controllability assumption}
In order to prove our structural stability result we need one further assumption which was not required in \cite{PSp15}, i.e. controllability of the nominal problem \eqref{eq:problema0} along the reference trajectory $\wxi$.
\begin{assumption}\label{ass:control}
$\wmu$ is the only adjoint covector associated to 
 $\wxi$.
\end{assumption}
The controllability assumption can in fact be stated in terms of the data of the nominal  problem \eqref{eq:problema0}.
For any $i=0, 1, 2$, let $\wt f_i$ be the pull-back of $f_i$ along the reference flow from the double switching time $\wtau$ to time $0$:
\begin{equation*}
\wt f_i (x) := \wS{\wtau*}^{-1}f_i \circ \wS{\wtau}(x) = \exp(-\wtau h_1)_* f_i \circ \exp \wtau h_1(x).
\end{equation*}
\begin{lemma}\label{le:control}
Assumption \ref{ass:control} holds if and only if 
\begin{equation*}
 \Span \left\{
T_{\wxo}N_0, \ T_{\wxo} \wh N_f, \wt f_0 (\wxo)
, \wt f_1 (\wxo), \wt f_2 (\wxo)\right\} = \Rn.
\end{equation*}
\end{lemma}
\begin{proof}
For ease of notation set $C := \Span \left\{
T_{\wxo}N_0, \ T_{\wxo} \wh N_f, \wt f_0 (\wxo)
, \wt f_1 (\wxo), \wt f_2 (\wxo)\right\}$. 

1. Let Assumption \ref{ass:control} hold and assume, by contradiction, that
$C \neq \Rn$. Then there exists $p \in C^\perp$, $p \neq 0$:
\begin{equation*}
\scal{p}{\dx} = 0 \quad \forall\dx\in T_{\wxo}N_0 \cup T_{\wxo}\wh N_f, \quad 
\scal{p}{\wt f_i(\wxo)} = 0 \quad \forall i=0, 1, 2.
\end{equation*}
Let $\mu(t) := \left( \wpo + p \right) \wSinv{t*} = \wmu(t) + p  \wSinv{t*} $. \\
If $t \in [0, \wtau]$ then  $\scal{p\wSinv{t*}}{h_1(\wxi(t))} = \scal{p}{\left( 
\wt f_0 - \wt f_1 - \wt f_2 \right)(\wxo)} = 0$. 
If $t \in (\wtau, T]$, then  $\scal{p\wSinv{t*}}{ h_2(\wxi(t))} = \scal{p\wSinv{\wtau*}}{ h_2(\wxi(\wtau))} = \scal{p}{\left( 
\wt f_0 + \wt f_1 + \wt f_2 \right)(\wxo)}  = 0$. \\
As $\left. \mu(0) \right\vert_{T_{\wxo}N_0 \cup T_{\wxo}\wh N_f} = \left. \wpo \right\vert_{T_{\wxo}N_0 \cup T_{\wxo}\wh N_f}$, it is easily checked that  $\lambda(t) := \big( \mu(t), \wxi(t)\big)$ satisfies PMP, a contradiction.

2. Assume $C = \Rn$ and suppose, by contradiction, there exists an adjoint covector $\mu(t)$ which, together with the reference triple $\reftri$ satisfies PMP.
Thus the following conditions hold:
\begin{align}
\begin{split}
& \scal{\mu(t)}{f_0(\wxi(t))} + \wh u_1(t)\scal{\mu(t)}{f_1(\wxi(t))}  + \wh u_2(t)\scal{\mu(t)}{f_2\wxi(t))} = \\
& = F_0 ( \mu(t), \wxi(t) ) + \abs{ F_1 ( \mu(t), \wxi(t) )}  + \abs{ F_2 ( \mu(t), \wxi(t) )} = p_0 \in \{0, 1\}; 
\end{split} \\
& \exists p \in \left( T_{\wxo}N_0 \right)^\perp \cap \left( T_{\wxo}\wh N_f \right)^\perp \colon \mu(t) = p\wSinv{t*}. \label{eq:passurdo}
\end{align}
As in $t = \wtau$ the double switch of $\wh u$ occurs,  we have
\[\scal{\mu(\wtau)}{f_1(\wxi(\wtau))} = \scal{\mu(\wtau)}{f_2(\wxi(\wtau))} = 0,
\]
so that $\scal{\mu(\wtau)}{f_0(\wxi(\wtau))} = p_0$, that is:
\begin{equation}
\scal{p}{\wt f_1(\wxo)} = \scal{p}{\wt f_2(\wxo)} = 0, \quad 
 \scal{p}{\wt f_0(\wxo)} = p_0.  \label{eq:passurdo0}
\end{equation}
We now distinguish between two cases:

1. if $\big(\mu(t), \wxi(t)\big)$ is an abnormal extremal ($p_0 = 0$) then, by \eqref{eq:passurdo} and \eqref{eq:passurdo0}, $p \in C^\perp$. As $C=\Rn$ this means that $p =0$, so that $\mu(t) \equiv 0$, a contradiction in PMP. 

2. if $\big(\mu(t), \wxi(t)\big)$ is a normal extremal ($p_0 = 1$) then, by \eqref{eq:passurdo} and \eqref{eq:passurdo0}, $p$ acts on $C = \Rn$ in the same way as $\wh p$, so that $p = \wh p$ and $\mu(t) = \wmu(t)$, i.e.~$\wmu$ is the only adjoint covector associated to $\wxi$.
\end{proof}

\section{The main result}
We are now in a position to state the main results of this paper, which will be proved in the following sections.
\begin{theorem}\label{thm:main1}
Under Assumptions \ref{ass:PMP}-\ref{ass:control}
there exists $\wt R \in (0, R)$ such that for any $r \in \ballt$, problem \eqref{eq:problemar} has a bang-bang state-local minimiser $\left(\Tr, \ur , \xir \right)$. Each control component of $\ur$ has exactly one switching time.
Let $\taur{i}$ be the switching time of $\ur_i$, $i=1, 2$. At time $\taur{i}$ the control component $\ur_i$ switches from the value $-1$ to the value $1$. The final time $\Tr$ and the switching times $\taur{1}$, $\taur{2}$ depend smoothly on $r$.
\end{theorem}
\begin{remark}
Notice that the switching times $\taur{1}$, $\taur{2}$ in Theorem \ref{thm:main1} may either coincide or be different, i.e.~ we may either have a double switching time or two simple switching times.
\end{remark}
\begin{theorem} \label{thm:main2}
Under Assumptions \ref{ass:PMP}-\ref{ass:control}
there exists $\wt R \in (0, R)$, $\e>0$ and a neighborhood $\cV$ of the graph of $\wla$ in $\R \times T^*\Rn$ such that for any $r \in \ballt$, the extremal pair $\lambda^ r$ associated to the local minimum time triple  $\left(\Tr, \ur , \xir \right)$ of Theorem \ref{thm:main1} is the only extremal pair whose final time is in $[\wh T - \e, \wh T + \e]$ and whose graph is in $\cV$.
\end{theorem}

\subsection{The coercivity of the second order variations}
\label{sec:coerc}
In \cite{PSp15}, in order to prove the strong local optimality result, the authors consider the bilinear form $Q_\nu$ associated to $J\se_\nu$, $\nu = 1,2$, i.e.~if $\de = (\dx, a_{1}, b, a_{2})$, $\df = (\dy, c_{1}, d, c_{2}) \in V_{0,\nu}$ then
\begin{equation*}
  \begin{split}
    Q_\nu[\de, \df] = & \uD^2(\a + \wbe )(\wxo)(\dx, \dy) %
   + \liedede{\dy}{(a_{1}g_{1} + b \, j_{\nu} + a_{2}g_{2})}{\wbe}{\wxo} %
   \\ %
   & + \liedede{\dx}{(c_{1}g_{1} + d \, j_{\nu} + c_{2}g_{2}) }{\wbe}{\wxo}  \\ %
   & + \liedede{(c_{1}g_{1} + d \, j_{\nu} + c_{2}g_{2})}{(a_{1}g_{1} + b \,
   j_{\nu} + a_{2}g_{2})}{\wbe}{\wxo} \\ %
   & + d a_1 \liede{\liebr{g_{1}}{j_{\nu}}}{\wbe}{\wxo} %
   + c_2 a_1 \liede{\liebr{g_{1}}{g_{2}} }{\wbe}{\wxo} %
   + c_2 b \liede{\liebr{j_{\nu}}{g_{2}}}{\wbe}{\wxo}. %
  \end{split}
\end{equation*}
The bilinear forms $Q_\nu$ can be written in a more compact way by introducing the linear Hamiltonians 
\begin{equation*}
\begin{split}
& G_i\se \colon (\dep, \dx) \in \Rns \times \Rn \mapsto
\scal{\dep}{g_i(\wxo)} + \liedede{\dx}{g_i}{\wbe}{\wxo} \in \R , \\ %
& J_\nu\se \colon (\dep, \dx) \in \Rns \times \Rn \mapsto
\scal{\dep}{j_\nu(\wxo)} + \liedede{\dx}{j_\nu}{\wbe}{\wxo} \in \R,  %
\end{split}
\end{equation*}
and the associated constant Hamiltonian vector fields
$\vG{1}\se$ and $\vJ{\nu}\se$.
An easy computation shows that
\begin{alignat*}{2}
&\dueforma{(\dep, \dx)}{\vG{1}\se} = G_i\se (\dep, \dx), \qquad
&& \dueforma{(\dep, \dx)}{\vJ{\nu}\se} =J_\nu\se (\dep, \dx), \\
& G_i\se(\vG{j}\se) = \liede{\liebr{g_j}{g_i}}{\wbe}{\wxo}, 
&& G_i\se(\vJ{\nu}\se) = \liede{\liebr{j_\nu}{g_i}}{\wbe}{\wxo}
= - J_\nu\se(\vG{i}\se).
\end{alignat*}
With these equalities at hand it is just a straightforward computation to prove the following proposition.
\begin{proposition}
For any admissible variation $\de = (\dx, a_1, b, a_2 ) \in V_\nu$ and any $\dep \in \Rns$ let
\begin{equation*}
(\dep_T, \dx_T) := (\dep, \dx) + a_1\vG{1}\se + b \vJ{\nu}\se + a_2\vG{2}\se.
\end{equation*}
Then 
\begin{equation*}
  \begin{split}
    Q_\nu[\de, \df] =  \uD^2(\a + \wbe )(\wxo)(\dx, \dy) 
    +\scal{\dep}{\dy} - \scal{\dep_T}{\dy + c_1 g_1 + d j_\nu + c_2 g_2} \\%
 + c_1 G_1\se\left(\dep, \dx \right) + d J\se_\nu\left((\dep, \dx)+ a_1\vG{1}\se\right) %
+ c_2 G_2\se\left((\dep, \dx)+ a_1\vG{1}\se + d\vJ{\nu}\se\right) .
  \end{split}
\end{equation*}
\end{proposition}
\begin{proposition}\label{prop:nucleo}
An admissible variation $\de \in V_{0,\nu}$ is in $V_{0,\nu}  \cap V_{0,\nu}^{\perp_{J\se_\nu}} $ if and only if there exists $\dep \in \Rns$ such that
\begin{align*}
& \dep = - \uD^2(\a + \wbe )(\wxo)(\dx, \cdot) + \omega_0, \qquad \omega_0 \in \left( T_{\wxo}N_0 \right)^\perp, \\ %
& G_1\se(\dep, \dx) = \dueforma{(\dep, \dx)}{\vG{1}\se} = 0 , \\ %
& J_\nu\se\left( (\dep, \dx) + a_1 \vG{1}\se\right) =\dueforma{(\dep, \dx) + a_1 \vG{1}\se}{\vJ{\nu}\se} = 0 , \\ %
& G_2\se\left( (\dep, \dx)  + a_1 \vG{1}\se + b\vJ{\nu}\se
\right)=\dueforma{(\dep, \dx) + a_1 \vG{1}\se + b\vJ{\nu}\se}{\vG{2}\se} = 0 , \\ %
& \dep_T \in \left( T_{\wxo}\wh N_f \right)^\perp  .
\end{align*}
\end{proposition}
\begin{corollary}\label{cor:nucleo}
Assume the coercivity assumption, Assumption \ref{ass:coerc}, holds and let $\de = \left( \dx, a_1, b, a_2 \right) \in V_{0,\nu}$. If there exists $\dep \in \Rns$ such that 
\begin{align*}
& \dep = - \uD^2(\a + \wbe )(\wxo)(\dx, \cdot) + \omega_0, \qquad \omega_0 \in \left( T_{\wxo}N_0 \right)^\perp, \\ %
& G_1\se(\dep, \dx) = \dueforma{(\dep, \dx)}{\vG{1}\se} = 0 , \\ %
& J_\nu\se\left( (\dep, \dx) + a_1 \vG{1}\se\right) =\dueforma{(\dep, \dx) + a_1 \vG{1}\se}{\vJ{\nu}\se} = 0 , \\ %
& G_2\se\left( (\dep, \dx)  + a_1 \vG{1}\se + b\vJ{\nu}\se
\right)=\dueforma{(\dep, \dx) + a_1 \vG{1}\se + b\vJ{\nu}\se}{\vG{2}\se} = 0 , \\ %
& \dep_T \in \left( T_{\wxo}\wh N_f \right)^\perp  .
\end{align*}
then $\de= (0,0,0,0)$.
\end{corollary}
Consider the Lagrangian manifold of the initial transversality conditions 
\begin{equation*}
\Lambda_0 := \left\{
\ell =\ud\alpha(x) + \omega \colon x \in N_0, \ \omega \in \left( T_x N_0 \right)^\perp, \ H_1(\ell) = 1
\right\}
\end{equation*}
so that 
\begin{equation*}
T_{\lo}\Lambda_0  := \left\{ \dl = \ud\alpha_* \dx + \omega \colon \dx \in T_{\wxo}N_0, \ \omega \in \left( T_{\wxo}N_0 \right)^\perp, \  \dueforma{\dl}{\vH{1}(\lo)} = 0
\right\}.
\end{equation*}
Let $i \colon (\dep, \dx) \in \Rns \times \Rn \mapsto \dl := -\dep + \ud\,(-\wbe)\dx \in T^*{\Rn}$.
The map $i$ is an antisymplectic isomorphism,
\begin{equation*}
\begin{split}
& i \vG{1}\se= \vH{1}(\lo) = \vG{1}(\lo) = \cFref_{\wtau \, *}^{-1}\vH{1}\circ\cFref_{\wtau}(\lo), \\ 
& i \vG{2}\se= \vG{2}(\lo) = \cFref_{\wtau \, *}^{-1}\vH{2}\circ\cFref_{\wtau}(\lo) =  \cFref_{\wT \, *}^{-1}\vH{2}\circ\cFref_{\wT}(\lo) , \\ 
& i \vJ{\nu}\se= \vJ{\nu}(\lo) = \cFref_{\wtau \, *}^{-1}\vK{\nu}\circ\cFref_{\wtau}(\lo) \qquad \nu=1, 2,
\end{split}
\end{equation*}
and $T_\lo\Lambda_0 = i L_0\se $
where
\begin{multline*}
\qquad 
L_0\se := \left\{\left(\dep, \dx \right) \colon \dx \in T_{\wxo}N_0, \ \dep = -\uD^2(\a +\wbe)(\wxo)(\dx, \cdot) + \omega,  \right. \\  
\left. \omega \in \left( T_{\wxo}N_0 \right)^\perp , \ \dueforma{(\dep, \dx)}{\vG{1}\se} = 0
\right\}. 
\qquad 
\end{multline*}

\begin{lemma}\label{le:implicit}
Under Assumptions \ref{ass:PMP} to \ref{ass:control} there exist $\wt R \in (0, R)$, $\e > 0 $ and a neighborhood $\cO(\lo)$ of $\lo$ in $T^*\Rn$ such that for any $r \in \ballt$, there exists a unique bang-bang extremal pair $\lar = (\mur, \xir)$ of \eqref{eq:problemar} having the following properties:
\begin{enumerate}[parsep=0pt]
\item $\lar$ is a normal extremal and $\lar(0) \in \cO(\lo)$;
\item each component $\ur_i$, $i=1,2$ of the associated control $\ur= (\ur_1, u^r_2)$ has exactly one switching time $\taur{i}$;  $\taur{1}, \taur{2} \in [\wtau -\e, \wtau + \e]$; at time $\taur{i}$ the control component $\ur_i$ switches from the value $-1$ to the value $+1$;
\item $\Tr \in [\wh T - \e, \wh T + \e]$;
\item $\taur{1}$, $\taur{2}$, $\Tr$ and $\lar(0)$ depend smoothly on $r$,
\item the bang arcs are regular: for $i =1, 2 \quad 
u^r_i (t) F^r_i(\lar(t)) > 0 \quad \forall t \neq \taur{i},
$
\item each switching time is regular: $\left. \ddt u_i^r(t) F^r_i(\lar(t)) \right\vert_{t = \taur{i}\pm} > 0,  \quad i=1,2$. 
\end{enumerate}
\end{lemma}
\begin{proof}
We prove claims 1-4 applying the implicit function theorem: for $\nu=1, 2$ consider the following system of $2n + 3$ scalar equations in the unknowns $r\in \ball$, $\ell = (p,x) \in T^*\Rn$, $t_1$, $t_2$, $t_3 \in \R$:
\begin{subequations}\label{eq:Psi}
\begin{align}
& \ell \in \left( T_{\pi\ell}\Nor \right)^\perp \times \Nor, \label{eq:prima} \\
& \Hr{1}(\ell) - 1 = 0 , \\
& \Kr{\nu}\circ\exp t_1\vHr{1}(\ell) - 1 = 0 , \\
& \Hr{2}\circ\exp (t_2 - t_1 )\vKr{\nu}\circ\exp t_1\vHr{1}(\ell) -1  = 0 , \\
\begin{split}
& \exp(t_3 - t_2)\vHr{2} \circ\exp (t_2 - t_1 )\vKr{\nu}\circ\exp t_1\vHr{1}(\ell) \\ %
& \phantom{\exp(t_3 - t_2)\vHr{2} \circ}
\in \left( T_{\pi\exp(t_3 - t_2)\vHr{2} \circ\exp (t_2 - t_1 )\vKr{\nu}\circ\exp t_1\vHr{1}(\ell)}\Nfr \right)^\perp \times \Nfr.
\end{split}
\end{align}
\end{subequations}
The linearised equations with respect to 
$\left( \ell, t_1, t_2, t_3 \right) $ at $(r, \ell, t_1, t_2, t_3) = (0, \lo, \wtau, \wtau, \wT)$ are given by
\begin{subequations}\label{eq:linear0}
\begin{align}
& \dl = (\dep, \dx)  \in T_\lo\left(\left(T_{\wxo}N_0\right)^\perp \times N_0 \right) , \\
& \dueforma{\dl}{\vH{1}(\lo)} = 0 , \\
& \dueforma{{\exp\wtau\vH{1}}_*\dl + \dt_1\vH{1}(\ld)}{(\vK{\nu} - \vH{1})(\ld)} = 0 , \\
& \dueforma{{\exp\wtau\vH{1}}_*\dl - \dt_1(\vK{\nu} - \vH{1})(\ld) + \dt_2\vK{\nu}(\ld)}{( \vH{2} - \vK{\nu})(\ld)} = 0 , \\
\begin{split}
& \cFref_{\wT \, *}\dl + {\exp(\wT - \wtau)\vH{2}}_*
\left( \dt_1 \vH{1} + (\dt_2 - \dt_1)\vK{\nu} + (\dt_3 - \dt_2)\vH{2}
\right)(\ld) \\ 
& \qquad\qquad\qquad\qquad  \in T_{\lf}\left(\left(T_{\wxf}N_f\right)^\perp \times N_f \right) .
\end{split}
\end{align}
\end{subequations}
Notice that $\left. \a\right\vert_{N_0} \equiv 0$ so that $\ud\a(x) \in \left( T_x N_0 \right)^\perp$ for any $x \in N_0$. Hence, if in equation \eqref{eq:prima} we write $\ell = \left( p, x \right)$ we must have $p - \ud\a(x) \in \left( T_x N_0 \right)^\perp$ for any $x \in N_0$ so that, if $\dl = \left( \dep, \dx \right)$ we get $\dx \in   T_{\wxo} N_0 $ $\dep - \uD^2\a(\wxo)[\dx, \cdot] \in  \left( T_{\wxo} N_0 \right)^\perp$. Thus, taking the pull-back to time $t=0$, the homogeneous linear system \eqref{eq:linear0} admits a nontrivial solution if and only if 
there exists $\dl = (\dep, \dx) \in T_{\lo}T^*\Rn $, $\dt_1$, $\dt_2$, $\dt_3 \in \R$, with at least one of them being different from zero, such that 
\begin{subequations}\label{eq:linear2}
\begin{align}
& \dl = \ud\a_*\dx + \omega_0, \qquad \dx \in T_{\wxo}N_0 , \ \omega_0 \in T_{\wpo}\left( T_{\wxo}N_0 \right)^\perp , \displaybreak[0]\\ 
& \dueforma{\dl}{\vG{1}(\lo)} = 0 , \displaybreak[0]\\
& \dueforma{\dl + \dt_1\vG{1}(\lo)}{(\vJ{\nu} - \vG{1})(\lo)} = 0 , \displaybreak[0]\\
& \dueforma{\dl - \dt_1(\vJ{\nu} - \vG{1})(\lo) + \dt_2\vJ{\nu}(\lo)}{( \vG{2} - \vJ{\nu})(\lo)} = 0 , \displaybreak[0]\\
& \dxf := \dx + 
\left( \dt_1 g_{1} + (\dt_2 - \dt_1)j_{\nu} + (\dt_3 - \dt_2)g_{2}
\right)(\wxo) \in T_{\wxo}\wh N_f, \displaybreak[0]\\
\begin{split}
& \dl + 
\left( \dt_1 \vG{1} + (\dt_2 - \dt_1)\vJ{\nu} + (\dt_3 - \dt_2)\vG{2}
\right)(\lo) \\ 
& \qquad\qquad\qquad\qquad = \ud \, (-\wbe)_*\dxf + \omega_{f}, \qquad \omega_f \in T_{\wpo}\left( T_{\wxo}\wh N_f \right)^\perp .
\end{split}
\end{align}
\end{subequations}
Applying the anti symplectic isomorphism $i^{-1}$ and denoting $i^{-1}\dl = (\dep, \dx)$, equations \eqref{eq:linear2} can also be written as
\begin{subequations}\label{eq:linear3}
\begin{align}
& (\dep, \dx) = -\uD^2(\a + \wbe)(\wxo)(\dx, \cdot) + \omega_0, \qquad \dx \in T_{\wxo}N_0 , \ \omega_0 \in \left( T_{\wxo}N_0 \right)^\perp , \\
& \dueforma{(\dep, \dx)}{\vG{1}\se} = 0 , \\ %
& \dueforma{(\dep, \dx) + \dt_1 \vG{1}\se}{\vJ{\nu}\se} = 0 , \\ %
& \dueforma{(\dep, \dx) + \dt_1 \vG{1}\se + (\dt_2 - \dt_1)\vJ{\nu}\se}{\vG{2}\se} = 0 , \\ %
\begin{split}
& (\dep_T, \dx_T) :=  (\dep, \dx) + \dt_1 \vG{1}\se + (\dt_2 - \dt_1)\vJ{\nu}\se \\ %
& \phantom{(\dep_T, \dx_T) :=   (\dep, \dx) } + (\dt_3 - \dt_2) \vG{2}\se
\in \left( T_{\wxo}\wh N_f \right)^\perp \times T_{\wxo} \wh N_f.
\end{split}
\end{align}
\end{subequations}
Thus, by Proposition \ref{prop:nucleo}, the variation $(\dx, \dt_1, \dt_2 - \dt_1, \dt_3 - \dt_2)$ is in $V_{0,\nu} \cap V_{0,\nu}^\perp $. As $J\se_\nu$ is coercive on $V_{0,\nu}$ we can apply Corollary \ref{cor:nucleo} and we get $\dx =0$, $\dt_1 = \dt_2 = \dt_3 = 0$, so that $\dep = 0$ if and only if $\omega_0 = 0$.
By equations \eqref{eq:linear3}, 
\begin{multline*}
\omega_0 \in  \Span\left\{ T_{\wxo} N_0 , \ T_{\wxo} \wh N_f, \  g_1(\wxo), j_\nu(\wxo), g_2(\wxo) \right\}^\perp =  \\
= \Span\left\{ T_{\wxo} N_0 , \ T_{\wxo} \wh N_f, \  \wt f_0(\wxo), \wt f_1(\wxo), \ \wt f_2(\wxo) \right\}^\perp,
\end{multline*}
 thus the controllability assumption, Assumption \ref{ass:control} and Lemma \ref{le:control} yield the claim. 

Thus we can apply the implicit function theorem to system \eqref{eq:Psi}. For $r \in \ballt$ let  $(\ellor, \taur{1}, \taur{2}, \Tr)$, $\ellor = (p_0^r, \xor)$ be the solution of system \eqref{eq:Psi}. The piecewise smooth curve $\lar (t) =  (\mu^r(t), \xi^r(t)) $ defined by
\begin{equation*}
\left\{
\begin{alignedat}{2}
\left. 
\begin{aligned} 
\exp t \vHr{1}(\ellor) , \quad & t \in [0, \taur{1}] , \\
\exp (t  - \taur{1} )\vKr{1}\circ\exp \taur{1} \vHr{1}(\ellor)  , \quad & t \in [\taur{1}, \taur{2}] , \\
\exp(t - \taur{2})\vHr{2} \circ\exp (\taur{2} - \taur{1} )\vKr{1}\circ\exp \taur{1}\vHr{1}(\ellor)  , \quad & t \in [\taur{2}, \Tr] , \\
\end{aligned}
\right\}
&& \quad \text{if }\taur{1} <  \taur{2} \\
\left. 
\begin{aligned}
\exp t \vHr{1}(\ellor) , \quad & t \in [0, \taur{2}] , \\
\exp (t  - \taur{2} )\vKr{2}\circ\exp \taur{2} \vHr{1}(\ellor)  , \quad & t \in [\taur{2}, \taur{1}] , \\
\exp(t - \taur{1})\vHr{2} \circ\exp (\taur{1} - \taur{2} )\vKr{2}\circ\exp \taur{2}\vHr{1}(\ellor)  , \quad & t \in [\taur{1}, \Tr] , \\
\end{aligned}
\right\}
&& \quad \text{if }\taur{2} < \taur{1}, \\
\left.
\begin{aligned}
\exp t \vHr{1}(\ellor) , \quad & t \in [0, \taur{1}] , \\
\exp(t - \taur{2})\vHr{2} \circ\exp \taur{2}\vHr{1}(\ellor)  , \quad & t \in [\taur{1}, \Tr] , \\
\end{aligned}
\right\}
&& \quad \text{if }\taur{2} = \taur{1}
\end{alignedat}
\right.
\end{equation*}
is a normal extremal of problem \eqref{eq:problemar} and satisfies claims 1-4

We can now complete the proof by proving claims 5-6:  possibly restricting $\wt R$ and $\cO(\lo)$ we can assume, by continuity
\begin{align}
& \begin{alignedat}{2}
& \Fr{i}(\lar(t)) < 0 \qquad && \forall t \in [0, \wtau - \e] , \\
& \Fr{i}(\lar(t)) > 0 \qquad && \forall t \in [\wtau - \e, \Tr] , \\
\end{alignedat} \qquad i=1, 2, \notag 
\displaybreak[0] \\
& \begin{aligned}
& \dueforma{\vHr{1}}{\vKr{\nu}}(\lar(t)) > 0
 \\
& \dueforma{\vKr{\nu}}{\vHr{2}}(\lar(t))  >  0
  \\
\end{aligned} \qquad  \forall t \in [\wtau - \e, \wtau + \e], \qquad \nu =1, 2 ,
\label{eq:ineq2r}
\end{align}
By construction, $\lar$ is a normal Pontryagin extremal of \eqref{eq:problemar}.
We prove Claim 5 in the case when $\taur{1} < \taur{2}$. The other cases are analougous. For any $t \in (\wtau  - \e, \taur{1})$ there exists $\theta_1 \in (t, \taur{1})$ such that 
\begin{equation*}
\begin{split}
2 \Fr{1}(\lar(t)) &= 2 \Fr{1}(\lar(\taur{1})) + (t - \taur{1}) \dedt
{\left( 2 \Fr{1}\circ\lar \right)}(\theta_1) \\ 
&= (t - \taur{1}) \dueforma{\vHr{1}}{2\vFr{1}}(\lar(\theta_1)) 
= (t - \taur{1}) \dueforma{\vHr{1}}{\vKr{1}}(\lar(\theta_1)) 
\end{split}
\end{equation*}
which is negative by \eqref{eq:ineq2r}.
Analougously,  for any $t \in (\taur{1}, \taur{2}]$ there exists $\theta_2 \in (\taur{1}, t)$ such that 
\begin{equation}
\begin{split}
2 \Fr{1}(\lar(t)) &= 2 \Fr{1}(\lar(\taur{1})) + (t - \taur{1}) \dedt
{\left( 2 \Fr{1}\circ\lar \right)}(\theta_2) \\ 
&= (t - \taur{1}) \dueforma{\vKr{1}}{2\vFr{1}}(\lar(\theta_2)) 
= (t - \taur{1}) \dueforma{\vHr{1}}{\vKr{1}}(\lar(\theta_2)) 
\end{split} \label{eq:ineq3r}
\end{equation}
which is positive by \eqref{eq:ineq2r}.
Finally, if $t \in (\taur{2}, \wtau + \e)$ there exists $\theta_3 \in (\taur{2}, \wT + \e)$ such that
\begin{equation*}
\begin{split}
2 \Fr{1}(\lar(t)) &= 2 \Fr{1}(\lar(\taur{2})) + (t - \taur{2}) \dedt
{\left( 2 \Fr{1}\circ\lar \right)}(\theta_3) \\ 
&= 2 \Fr{1}(\lar(\taur{2})) + (t - \taur{2}) \dueforma{\vHr{2}}{2\vFr{1}}(\lar(\theta_3))  \\ 
& = 2 \Fr{1}(\lar(\taur{2})) + (t - \taur{2}) \dueforma{\vKr{2}}{\vHr{2}}(\lar(\theta_3)) 
\end{split}
\end{equation*}
which is positive by \eqref{eq:ineq2r} and \eqref{eq:ineq3r}.
The proof for the sign of $\Fr{2}(\lar(t))$ follows the same line.

Finally, the switching times $\taur{i}$ are regular (claim 6) thanks to inequalities \eqref{eq:ineq2r}.
\end{proof}
We can now prove Theorem \ref{thm:main1}, i.e.~we prove that projection $\xir$ of the extremal $\lar$ defined in Lemma \ref{le:implicit} is a state-local optimal trajectory for problem \eqref{eq:problemar}.
\begin{proof}[Proof of Theorem \ref{thm:main1}]
By construction and by Lemma \ref{le:implicit}, $(\Tr, \xir = \pi\lar, \ur)$ satisfies PMP in its normal form and the regularity assumptions for problem \eqref{eq:problemar}. Thus it suffices to prove that  $\xir$ has no self-intersection and that the second order variation associated to \eqref{eq:problemar} is coercive.

{\em Injectivity of $\xir$. } Assume by contradiction there exists a sequence $\{r_k\}_{k \in \N} \subset \ballt$
that converges to $0$ and such that there exist $t_{1, k}$, $t_{2, k}$, 
$0 \leq t_{1, k} < t_{2, k} \leq \Tk$, $\xik(t_{1, k}) = \xik(t_{2, k})$. 
Up to a subsequence both  $t_{1, k}$ and $t_{2, k}$ converge.
Let $\overline{t}_i := \lim_{k\to\infty}t_{i, k} \in [0, \wT]$, $i = 1, 2$.
If $\overline t_1 < \overline t_2$, then $\wxi(\overline t_1) = \wxi(\overline t_2)$, a contradiction.
Assume then $\overline t_1 = \overline t_2 =: \overline t$.
Different cases may occur:
\begin{enumerate}[ leftmargin=0mm, label=\bf\arabic*.]
\setlength\itemindent{6mm} 
\item Up to a subsequence $0 \leq t_{1, k} < t_{2, k} \leq \tauk{1}$. In this case
\begin{equation*}
0 = \xik(t_{2, k}) - \xik(t_{1, k}) = \int_{t_{1, k}}^{t_{2, k}}\hr{1}(\xik(s)) \ud s.
\end{equation*} 
Applying the mean value thorem componentwise we get:
\begin{equation}
\label{eq:inje1}
\forall j = 1, \ldots, n \qquad \exists s^k_j \in [t_{1, k} , t_{2, k}]  \colon (\hr{1})_j(\xik(s^k_j)) = 0.
\end{equation}
Thus, as $k \to \infty$ in \eqref{eq:inje1} we obtain $h_1(\wxi(\overline t)) = 0$, a contradiction since $\overline t \in [0, \wtau_1]$ and  $H_1(\wla(t)) = 1 \quad \forall t \in [0, \wtau_1]$.
\item Up to a subsequence $0 \leq t_{1, k} \leq \tauk{1} < t_{2, k} \leq \tauk{2}$. In this case 
\begin{multline}
\label{eq:inje2}
\qquad 0 = \dfrac{\xik(t_{2, k}) - \xik(t_{1, k})}{t_{2, k} - t_{1, k}} = 
\dfrac{\tauk{1} - t_{1, k}}{t_{2, k} - t_{1, k}} 
\fint_{t_{1, k}}^{\tauk{1}} \hr{1}(\xik(s)) \ud s + \\ 
+ \dfrac{t_{2, k}  - \tauk{1} }{t_{2, k} - t_{1, k}}  \fint_{\tauk{1}}^{t_{2, k}}\kr{1}(\xik(s)) \ud s. \qquad
\end{multline}
Up to a subsequence there exists $\displaystyle\lim_{k\to\infty}\dfrac{\tauk{1} - t_{1, k}}{t_{2, k} - t_{1, k}} = c \in [0,1]$, so that passing to the limit in \eqref{eq:inje2} we obtain 
\begin{equation*}
 0 = c \, h_1(\wxd) + (1 - c) \, k_1(\wxd) = f_0(\wxd) + (1 - 2c)\, f_1(\wxd) - f_2(\wxd).
\end{equation*}
A contradiction, since $f_0(\wxd) \notin \Span \left\{ f_1(\wxd), f_2(\wxd) \right\}$.
\item Up to a subsequence $0 \leq t_{1, k} \leq \tauk{1} \leq \tauk{2} \leq t_{2, k}$. In this case 
\begin{multline}
\label{eq:inje3}
\quad 0 = \dfrac{\xik(t_{2, k}) - \xik(t_{1, k})}{t_{2, k} - t_{1, k}} = 
\dfrac{\tauk{1} - t_{1, k}}{t_{2, k} - t_{1, k}} 
\fint_{t_{1, k}}^{\tauk{1}} \hr{1}(\xik(s)) \ud s + \\ 
+ \dfrac{\tauk{2}  - \tauk{1} }{t_{2, k} - t_{1, k}}  \fint_{\tauk{1}}^{t_{2, k}}\kr{1}(\xik(s)) \ud s 
+ \dfrac{t_{2, k} - \tauk{2}}{t_{2, k} - t_{1, k}}  \fint_{\tauk{2}}^{t_{2, k}} \hr{2}(\xik(s)) \ud s 
. \qquad
\end{multline}
Up to a subsequence there exist 
\begin{equation*}\displaystyle\lim_{k\to\infty}\dfrac{\tauk{1} - t_{1, k}}{t_{2, k} - t_{1, k}} = c_1 \in [0,1], \qquad
\displaystyle\lim_{k\to\infty}\dfrac{\tauk{2} - \tauk{1}}{t_{2, k} - t_{1, k}} = c_2 \in [0,1],
\end{equation*}
 so that passing to the limit in \eqref{eq:inje3} we obtain 
\begin{multline*}
\qquad  0 = c_1 \, h_1(\wxd) + c_2 \, k_1(\wxd)+ (1- c_1 - c_2) \, h_2(\wxd) = \\ 
= f_0(\wxd) + (1 - 2 c_1)\, f_1(\wxd) - (1 - 2c_1 - 2c_2)f_2(\wxd). \qquad 
\end{multline*}
A contradiction, since $f_0(\wxd) \notin \Span \left\{ f_1(\wxd), f_2(\wxd) \right\}$.
\end{enumerate}
In the other cases the proof follows the same line.

{\em Coercivity of the second variation. }  
Let $(\Tr, \lar, \ur)$ be the extremal defined in Lemma \ref{le:implicit}, let $\xir := \pi\lar$ and $\xor := \xir(0)$. Assume $\taur{1} < \taur{2}$. In this case the trajectory $\xir$ is driven by the dynamics
\begin{equation*}
\phi^r_t := \begin{cases}
\hr{1},  \quad & t \in [0, \taur{1}], \\ 
\kr{1},  \quad & t \in (\taur{1}, \taur{2}], \\ 
\hr{2},  \quad & t \in (\taur{2}, \Tr].
\end{cases}
\end{equation*}

Let $\Sr{t}$ be the flow at time $t$ associated to $\phi^r_t$ and consider the pull-back vector fields
\begin{equation*} 
\begin{alignedat}{2}
  \gr{1}(x) &:= \Srinv{t \, *} \hr{1} \circ \Sr{t}(x), \quad && t \in [0, \taur{1}],  \\ %
  \jr{1}(x) &:= \Srinv{t \, *} \kr{1} \circ \Sr{t}(x), \quad && t \in [\taur{1}, \taur{2}], \\
  \gr{2}(x) &:= \Srinv{t \, *} \hr{2} \circ \Sr{t}(x), \quad && t \in [\taur{2}, \Tr].
\end{alignedat}
\end{equation*}
Let $\ar$ be a function that vanishes on $\Nor$ and such that $\ud\a^r(\xir(0)) = \lar(0)$. Let $\br$ be a smooth function that vanishes on $\Srinv{\Tr}(\Nfr)$, such that $\ud\br(\xir(0)) = - \lar(0)$. 
Finally consider the linearisation of the constraints
\begin{equation*}
 V_0^r := \left\{
 \de = (\dx, a_1, b, a_2) \in T_{\xor}\Nor \times \R^3 \colon 
\dx + a_1 \gr{1} + b \jr{1} + a_2 \gr{2} \in T_{\xor}{\Nor}
  \right\} .
\end{equation*}
Then the second variation at the switching points, see e.g.~\cite{Pog06}, is given by
\begin{equation*}
\begin{split}
J\se_r[\de]^2 & =   \dfrac{1}{2} \uD^2 \left( \ar + \br \right)(\xor)[\dx]^2 
+ \liedede{\dx}{(a_1 \gr{1} + b \jr{1} + a_2 \gr{2})}{\br}{\xor} \\ 
& + \dfrac{1}{2}\liededo{(a_1 \gr{1} + b \jr{1} + a_2 \gr{2})}{\br}{\xor} + \dfrac{1}{2} a_1 b \, \liede{\liebr{\gr{1}}{\jr{1}}}{\br}{\xor} \\ 
& + \dfrac{1}{2} a_1 a_2 \, \liede{\liebr{\gr{1}}{\gr{2}}}{\br}{\xor}
+ \dfrac{1}{2} a_2 b \, \liede{\liebr{\jr{1}}{\gr{2}}}{\br}{\xor}.
\end{split}
\end{equation*}
We now show, with a contradiction argument,  that $J\se_r$ is coercive on $V_0^r$: assume there exists a sequence  $\{r_k\}_{k \in \N} \subset (0, \wt R)$
that converges to $0$ and such that $J\se_{r_k}$ is not coercive on $V_0^r$, i.e.~there exists $\de^k = \left( \dx^k, a_1^k, b^k, a_2^k
\right) \in V_0^r$ such that $\norm{\dx^k} + \abs{a_1^k} + \abs{b^k} +\abs{a_2^k} = 1$ and $J\se_r[\de^ k]^2 \leq 0$.
Up to a subsequence $\de^k$ converges to some $\overline\de = \left( 
\overline\dx, \overline a_1, \overline b, \overline a_2
\right) \in V_{0,1}$ and such that $\norm{ \overline\dx} + \abs{ \overline a_1} + \abs{ \overline b} +\abs{ \overline a_2} = 1$. Thus
\begin{equation*}
0 \geq \lim_{k\to\infty} J\se_r[\de^k]^2 = J\se[\overline\de]^2 > 0,
\end{equation*}
a contradiction. 

We have thus proved that $\left( \Tr, \xir, \ur \right)$, together with $\lar$ satisfies all the assumptions of Theorem 1 in \cite{Pog06}, so that $\xir$ is a state-locally optimal trajectory 
for problem \eqref{eq:problemar}. 
If $\taur{2} < \taur{1} $ the proof follows the same lines. 

Let us consider the case $\taur{1} = \taur{2} =: \taur{}$. In this case, as in the nominal problem \eqref{eq:problema0} we have to consider two different second order approximations and to prove that they are coercive on the respective half-space of linearised constraints.

The trajectory $\xir$ is driven by the dynamics
\begin{equation*}
\phi^r_t := \begin{cases}
\hr{1},  \quad & t \in [0, \taur{}], \\ 
\hr{2},  \quad & t \in (\taur{}, \Tr].
\end{cases}
\end{equation*}
Denoting again by $\Sr{t}$ the flow at time $t$ associated to $\phi^r$, we consider the pullback vector fields  
\begin{equation*} 
\begin{alignedat}{2}
  \gr{i}(x) &:= \Srinv{\taur{} \, *} \hr{1} \circ \Sr{\taur{}}(x),   \quad && i = 1, 2, \\ %
  \jr{\nu}(x) &:= \Srinv{\taur{} \, *} \kr{\nu} \circ \Sr{\taur{}}(x), \quad && \nu = 1, 2.
\end{alignedat}
\end{equation*}
Let $\ar$, $\br$ and $\gamma^r$ be as before. Then the linearisation of the constraints is given by the half spaces
\begin{multline*}
 V^{+, r}_{\nu} := \left\{ \de = (\dx, a_{1}, b, a_{2}) \in T_{\xor}\Nor \times \R \times \R^{+} \times \R \colon \right. \\ 
\left. \dx + a_{1}\gr{1}(\wxo) + b \, \jr{\nu}(\wxo)
  + a_{2}\gr{2}(\wxo) \in T_{\xor} \Srinv{\Tr}(\Nfr) \right\},
  \quad \nu =1, 2  \quad
\end{multline*}
and the second order approximation is given by
\begin{equation*}
\begin{split}
{J\se_{\nu}}^{, r}  [\dx, a_{1}, b, a_{2}] = & \uD^{2} \left( \ar + \br \right)(\xor)[\dx]^{2} + 2 
\liedede{\dx}{ (a_{1}\gr{1} + b \, \jr{\nu} + a_{2}\gr{2})}{ \br}{\xor}  \\ 
&+ \liededo{\left( a_{1}\gr{1} + b \, \jr{\nu} + a_{2}\gr{2} \right) }{\br}{\xor} 
+ a_{1}b \liede{\liebr{\gr{1}}{\jr{\nu}}}{ \br}{\xor}     \\ 
& + a_{1}a_{2} \liede{\liebr{\gr{1}}{\gr{2}} }{ \br }{\xor } + b \, a_{2} \liede{\liebr{\jr{\nu}}{\gr{2}}}{ \br}{\xor}.
\end{split}
\end{equation*}
With the same contradiction argument used in the previous case it is easy to show that ${J\se_{\nu}}^{, r} $ is coercive on $ V^{+, r}_{\nu}$, $\nu=1,2$. Thus $\left( \Tr, \xir, \ur \right)$, together with $\lar$ satisfies all the assumptions of Theorem 4.2 in \cite{PSp15}, so that $\xir$ is a state-locally optimal trajectory for problem \eqref{eq:problemar}. 
\end{proof}

\section{Local uniqueness}
We now prove the local uniqueness of the extremal $\lar$ in the cotangent bundle $T^*\R^n$, namely we prove Theorem \ref{thm:main2}.
 The proof is carried out by showing that there exists a tubular neighborhood $\cV$ in $\R \times T^*\Rn$ of the graph of $\wla$ such that, if $\widetilde\la \colon [0, \widetilde T] \to T^*\Rn$ is an extremal whose graph is in $\cV$, with $\widetilde T$ close to $\wh T$,  then the associated control $\widetilde u= \left(\widetilde u_1, \widetilde u_2 \right)$ is bang-bang  and each control component switches once and only once from the value $-1$ to the value $1$. This implies that $\widetilde\la$  satisfies system \eqref{eq:Psi} which, by the implicit function theorem, admits one and only one solution, i.e.~$\widetilde\la = \la^ r$. 

By the regularity assumption at the switching time (Assumption \ref{ass:3}) and by continuity, there exists $\overline\delta > 0$ such that
  \begin{equation*} 
 \scal{\wla(t)}{\liebr{f_0}{f_i} (\wxi(t))} > 
 \abs{\scal{\wla(t)}{\liebr{f_1}{ f_2} (\wxi(t))}} ,
    \quad  \forall t \in [\wtau - \overline\delta, \wtau + \overline\delta], \quad i = 1,2 . 
  \end{equation*}
For any $\delta \in (0, \overline{\delta}]$
and $i=1, 2$ define
\begin{alignat*}{3}
& \a^a_i(\delta) &&= \min &&\left\{
u_i(t)F_i \circ \wla(t) = - F_i \circ \wla(t) \colon t \in [0, \wtau - \delta ] 
\right\} , \\
& \a^p_i(\delta) &&= \min &&\left\{
u_i(t)F_i \circ \wla(t) = F_i \circ \wla(t) \colon t \in [\wtau + \delta, \wT ] 
\right\} ,
\end{alignat*}
and let 
\begin{equation*}
m(\delta) := \min\left\{
\dueforma{\vF{0}}{\vF{1}}(\wla(t)) - \abs{\dueforma{\vF{1}}{\vF{2}}(\wla(t))}, \ i=1, 2, \ t \in [\wtau - \delta, \wtau + \delta]
\right\}.
\end{equation*}
By continuity there exists $\cO(\lo) \subset T^*\Rn$ such that 
\begin{alignat*}{2}
& F_i \circ \cFref_t(\ell) < \dfrac{- \, \a^a_i(\delta)}{2} \qquad && \forall (t, \ell) \in [0, \wtau - \delta ] \times \cO(\lo), \\
& F_i \circ \cFref_t(\ell) > \dfrac{ \a^p_i(\delta)}{2} && \forall (t, \ell) \in  [\wtau + \delta, \wT ]\times \cO(\lo), \\
& \dueforma{\vF{0}}{\vF{1}}(\cFref_t(\ell)) - \abs{\dueforma{\vF{1}}{\vF{2}}(\cFref_t(\ell))} > \dfrac{m(\delta)}{2}\quad && \forall (t, \ell) \in  [\wtau - \delta, \wT + \delta ]\times \cO(\lo),
\end{alignat*}
and, again by continuity, there esists $\ov R > 0$ such that 
\begin{alignat}{3}
& \Fr{i} \circ \cFref_t(\ell) < \dfrac{- \, \a^a_i(\delta)}{4} \qquad && \forall (t, \ell) \in [0, \wtau - \delta ] \times \cO(\lo), \quad && \forall r \colon  \abs{r} \leq \ov R , \label{eq:Fra} \\
& \Fr{i} \circ \cFref_t(\ell) > \dfrac{ \a^p_i(\delta)}{4} && \forall (t, \ell) \in  [\wtau + \delta, \wT ]\times \cO(\lo), \quad && \forall r \colon \abs{r} \leq \ov R , \label{eq:Frp}
\end{alignat}
and
\begin{equation}
\begin{split}
\dueforma{\vFr{0}}{\vFr{1}}(\cFref_t(\ell)) -
 \abs{\dueforma{\vFr{1}}{\vFr{2}}(\cFref_t(\ell))} > \dfrac{m(\delta)}{4} \\ 
   \forall (t, \ell) \in  [\wtau - \delta, \wT + \delta] & 
   \times \cO(\lo), \quad  \forall r \colon \abs{r} \leq \ov R .
\end{split}
\end{equation}
Let $\wt\la \colon [0, \wt T] \to T^*\Rn$ be an extremal of \eqref{eq:problemar} whose graph is in the tubular set
\begin{equation*}
\cV_\delta = \left\{
(t, \cFref_t(\ell)) \colon t \in [0, \wh T + \delta], \ \ell \in \cO_\delta(\lo)
\right\}
\end{equation*}
and such that $\abs{\wt T -\wh T} < \delta$. \\
By \eqref{eq:Fra}-\eqref{eq:Frp}, for $i=1, 2$, 
\begin{equation*}
 \Fr{i} \circ \wt\la(t) < \dfrac{- \, \a^a_i}{4} \quad \forall t \in [0, \wtau -\delta], \qquad 
 \Fr{i} \circ \wt\la(t) > \dfrac{ \a^p_i}{4} \quad \forall t \in [\wtau +\delta, \wt T]
\end{equation*}
hence there exists $\wt t_i \in (\wtau - \delta, \wtau + \delta)$ such that $\Fr{i} \circ \wt\la(\wt t_i) = 0$.
We now prove that $\wt t_i$ is the only time at which $\Fr{i}\circ \wt \la$ is zero. More precisely we show that  $\Fr{i}\circ \wt \la(t) $ is strictly monotone increasing in the interval $[\wtau - \delta, \wtau + \delta]$ . Let $\wtau - \delta \leq s_1 < s_2 \leq \wtau + \delta$:
\begin{multline*}
\Fr{i} \circ \wt\la(s_2) - \Fr{i} \circ \wt\la(s_1) = 
 \int_{s_1}^{s_2} \dds\Fr{i}\circ\wt\la(s) \ud s =  \\ 
= \int_{s_1}^{s_2} \dueforma{\vFr{0} + \wt u_1(s)\vFr{1} + \wt u_2(s) \vFr{2}}{\vFr{1}}(\wt\la(s)) \ud s = \\ 
 = \int_{s_1}^{s_2} \left( \dueforma{\vFr{0}}{\vFr{1}} - \wt u_2(s)\dueforma{\vFr{1}}{\vFr{2}}   \right) (\wt\la(s)) \ud s  > (s_2 - s_1)\dfrac{m(\delta)}{4}.
\end{multline*}
Thus each component of the control $\wt u$  associated to $\wt\xi := \pi\wt\la$ switches once and only once from the value $-1$ to the value $+1$.

\bibliographystyle{plain}
\bibliography{bibliocompleta}
\end{document}